\documentclass[12pt,english]{article}
\usepackage[T1]{fontenc}
\usepackage[latin9]{inputenc}
\usepackage{amsmath}
\usepackage{amsthm}
\usepackage{amssymb}
\usepackage{enumerate}
\usepackage{color}
\numberwithin{equation}{section}
\makeatletter

\theoremstyle{plain}
\theoremstyle{remark}

\usepackage{fullpage}
\usepackage{makeidx}
\usepackage{amsfonts}
\usepackage{latexsym}
\usepackage{babel}
\usepackage{babel}
\makeatother

\allowdisplaybreaks 

\def \1{{\bf 1}}
\def\C{{\mathbb{C}}}
\def\Z{{\mathbb{Z}}}

\def\N{{\mathbb{N}}}

\theoremstyle{definition}
\newtheorem{lemma}{Lemma}[section]
\newtheorem{theorem}[lemma]{Theorem}
\newtheorem{corollary}[lemma]{Corollary}
\newtheorem{proposition}[lemma]{Proposition}
\newtheorem{definition}[lemma]{Definition}

\newtheorem{example}[lemma]{Example}
\newtheorem{remark}[lemma]{Remark}
\usepackage{times}

\usepackage[blocks]{authblk}

\title{Hopf actions on  vertex algebras}

\author{Chongying Dong  \footnote{supported by the Simons foundation  634104 }}
\affil{Department of Mathematics, University of
California, Santa Cruz, CA 95064 USA}

\author{Li Ren \footnote{partially supported by NSFC grant 12071314}}
\affil{School of Mathematics,  Sichuan University,
	Chengdu 610064, China}

\author{Chao Yang \footnote{ The author is supported by the NSFC No. 12301039} }
\affil{School of Mathematics,  Southwest Jiaotong University,
	Chengdu 611756, China }

\begin{document}
\maketitle
\abstract

In this article, we investigate Hopf actions on vertex algebras.  
Our first main result is that 
every finite-dimensional Hopf algebra that inner faithfully acts on a given $\pi_2$-injective vertex algebra must be a group algebra.
Secondly, under suitable assumptions, we establish  a Schur-Weyl type duality for semisimple Hopf actions on Hopf modules of vertex algebras.

\section{Introduction}
This paper is a continuation of \cite{DW}.  We study Hopf actions on vertex algebras, and establish Schur-Weyl type dualities for semisimple Hopf actions on vertex algebras and their modules.

A systematic study of  Hopf actions on vertex operator algebras was first initiated in  \cite{DW}. It was proved that any finite-dimensional Hopf algebra that can faithfully act on a simple vertex operator algebra is the group algebra of some finite automorphism group of the vertex operator algebra. Moreover,   a Schur-Weyl type duality for semisimple Hopf actions on vertex operator algebras was obtained.

The concept of Hopf actions on vertex  algebras is a natural extension of the  concepts of
group and Lie algebra actions on vertex algebras.  Given a vertex  algebra $V$ with an action of a Hopf algebra $H$,  the fixed point subspace $V^H$ is a vertex  algebra. There are two central problems in Hopf action on vertex  algebra: 1) Determine what kind of Hopf algebra can act on a vertex algebra, 2) Understand  the structural and representation theory of $V^H.$  The purpose of this paper is to give some partial answers to these two problems.  Specifically, we prove that a finite-dimensional Hopf algebras that act on a $\pi_2$-injective vertex algebra is a group algebra, and establish a Schur-Weyl type duality for semisimple Hopf actions on Hopf modules of vertex algebras, improving and extending the results in \cite{DW} to the case of vertex algebras.

Note that any Hopf action on a vertex algebra can be uniquely transformed into an inner faithful Hopf action, while keeping the fixed point subalgebra $V^H$ unchanged. An inner faithful Hopf action means that the action does not factor through a smaller Hopf algebra. 
In order to  understand of the fixed point subalgebras $V^H$ under Hopf action, it is good enough to consider the inner faithful Hopf action.
Our first main result says that a finite-dimensional Hopf algebra which has an  inner faithful action  on a $\pi_2$-injective vertex algebra must be a group algebra. In general, we expect  that every Hopf algebra $H$ which has an inner faithful action on a $\pi_2$-injective vertex algebra is cocommutative.
This expectation is always true if $H$ is finite dimensional from the proof of Theorem \ref{thm-gr-alg}.
If this expectation is confirmed, then, based on the structure theory of cocommutative Hopf algebras (see Theorem 3.8.2 of  \cite{M}), 
understanding the structure and representation theory of fixed point subalgebras under Hopf actions
can be effectively reduced to studying the structure and representation theory of fixed point subalgebras under the actions of groups and Lie algebras.

We point out  that $\pi_2$-injective vertex algebras are a very broad class of vertex algebras, 
which include all simple vertex algebras of countable dimension and all nondegenerate vertex algebras defined in \cite{EK,Li1}. 

Now, we turn our attention to  the Schur-Weyl type duality.
Various versions of Schur-Weyl type duality can be found in \cite{DM,DLM1,DW,DRY,DY,DYa,MT,Ta,YaY}.
Another motivation for our work is to extend the Schur-Weyl type duality for semisimple Hopf actions on vertex operator algebras in \cite{DW} to the case of vertex algebras.
Let us state the result more explicitly:
Assume that $V$ is an irreducible vertex algebra of countable dimension,
and $H$ is a Hopf algebra such that $V$ is an $H$-module vertex algebra.
Let $M$ be a Hopf $V$-module such that $M$ is an irreducible $V$-module.
Additionally, we assume that both $V$ and $M$ are direct sums of finite-dimensional irreducible $H$-modules.
Our second main result in this paper  is the following Schur-Weyl type duality:
$$M=\bigoplus_{i \in I}W_i \otimes M_i,$$
where $\{W_i \ | \ i \in I \}$ is the set of all finite-dimensional inequivalent irreducible $H$-modules appearing in $M$,
and $M_i$ is the multiplicity space of $W_i$ in $M$.
Moreover, $M_i$ is an irreducible $V^H$-module and $M_i$ and $M_j$ are isomorphic $V^H$-modules if and only  $i=j$.
This result generalizes the  previous Schur-Weyl type duality for groups and Lie algebras actions on vertex (operator) algebras.

We should mention that the main tool used in \cite{DW} is the associative algebra $A_n(V)$ for $n\geq 0$ \cite{Z, DLM2} for studying $\Z_+$-graded $V$-modules. Since a vertex algebra itself is not even graded, we cannot use $A_n(V)$ theory in our proof.  
So in this paper we use the ideas and methods developed  in \cite{DRY} to study the Hopf actions on vertex algebras and prove the desired results.


This paper is organized as follows: In Section 2, we review the foundations of Hopf algebras and vertex algebras.
In Section 3, we introduce the concepts and basic properties of Hopf actions on vertex algebras.
In Section 4, we introduce the concept of $\pi_2$-injective vertex algebras and provide examples of such algebras.
In Section 5, we discuss Hopf actions on $\pi_2$-injective vertex algebras and
prove the main results of Theorem \ref{thm-gr-alg} and Theorem \ref{thm-Hopf-ideal}.
In Section 6, we establish a Schur-Weyl type duality for semisimple Hopf actions on Hopf modules of vertex algebras.


\section{Preliminaries}

Throughout this paper, we work over the complex field $\C$.
The unadorned symbol $\otimes$ means the tensor product over $\C$.
We denote by $\N$ the set of nonnegative integers.


\subsection{Vertex algebras}

\begin{definition} (\cite{B, LL})
A {\em vertex algebra} is a triple  $(V, {\bf 1}, Y(z) )$   consisting of a vector space $V$,  
a distinguished vector ${\bf 1} \in V$ called the {\em vacuum vector},
and a linear map
$$Y(z): \ V \to \text{End}_{\C}(V)[[z, z^{-1}]], \ \ \ v \to Y(v,z)=\sum_{n \in \Z}v_nz^{-n-1} \ \ (v_n \in \text{End}_{\C}(V) ), $$
satisfying the following conditions:
\begin{enumerate}[{(1)}]
\item For given $u, v \in V$, $ u_nv=0$ \ when  $n$ is sufficiently large;  

\item  $ Y({\bf 1},z)=Id_{V}; $

\item For $u \in V$, we have $ Y(u,z){\bf 1}= u+ (u_{-2}{\bf 1})z + (u_{-3}{\bf 1}) z^2 + \cdots  \in V[[z]];$

\item the Jacobi identity holds for $u, v \in V$,
\begin{align*}
& \displaystyle{z^{-1}_0\delta\left(\frac{z_1-z_2}{z_0}\right)Y(u,z_1)Y(v,z_2)
-z^{-1}_0\delta\left(\frac{z_2-z_1}{-z_0}\right)Y(v,z_2)Y(u,z_1)}\\
& \displaystyle{=z_2^{-1}\delta \left(\frac{z_1-z_0}{z_2}\right) Y(Y(u,z_0)v,z_2)}.
\end{align*}
\end{enumerate}

In the Jacobi identity, $\delta(z)=\sum_{n \in \Z } z^n$ is the formal delta function, and 
all binomial expressions (here and below) are to be expanded in nonnegative integral powers of the second variable.

\end{definition}

\begin{definition}
Let $V$ be a vertex algebra.
A $V$-module is a vector space $M$ equipped with a linear map
$$Y_M(z): \ V \to \text{End}_{\C}(M)[[z, z^{-1}]], \ \ \ v \to Y_M(v,z)=\sum_{n \in \Z}v_nz^{-n-1} \ \ (v_n \in \text{End}_{\C}(M) ), $$
satisfying the following axioms:
\begin{enumerate}[{(1)}]
\item For $u \in V$ and $w \in M$, $u_n w =0$ for sufficiently large $n$;  

\item $Y_M({\bf 1},z)=Id_M$;

\item  the Jacobi identity holds for $u, v \in V$,
\begin{align*}
& \displaystyle{ z^{-1}_0\delta\left(\frac{z_1-z_2}{z_0}\right)
Y_M(u,z_1)Y_M(v,z_2)-z^{-1}_0\delta\left(\frac{z_2-z_1}{-z_0}\right)Y_M(v,z_2)Y_M(u,z_1)}\\
& \displaystyle{=z_2^{-1}\delta\left(\frac{z_1-z_0}{z_2}\right) Y_M(Y(u,z_0)v,z_2)}.
\end{align*}
\end{enumerate}

\end{definition}

\begin{example}
Let $(A, \partial)$ be a commutative differential algebra with a unit element $\bf 1$,
i.e. $A$ is a commutative associative algebra  and  $\partial$ is a derivation of $A$.
For $a , b \in A,$ we define
$$Y^{(A,\partial)}(a,z)b=(e^{z\partial}a)b=\sum_{n=0}^{\infty}\frac{1}{n!}(\partial^na)bz^n.$$
Then $(A, Y^{(A,\partial)}( z), {\bf 1})$ forms a commutative vertex algebra \cite{B,LL}.
If there is no ambiguity, we still use $(A, \mathcal{\partial})$ to denote the commutative vertex algebra $(A, Y^{(A,\partial)}( z), {\bf 1})$.
\end{example}

The following results are direct consequences of definition (see e.g. \cite{LL}).
\begin{proposition} \label{ED-der}
Let $V$ be a vertex algebra. Let $\mathcal{D}$ be the endomorphism of $V$ defined by
$\mathcal{D}(v)=v_{-2}\1 $  for $v \in V$. Then the following identities hold:
\begin{enumerate}[{(1)}]
\item $Y(u,z){\bf 1}=e^{z\mathcal{D}}{\bf 1};$

\item $Y(u, z)v =e^{z \mathcal{D}}Y(v,-z)u;$

\item $e^{z_0\mathcal{D}}Y(u,z)e^{-z_0\mathcal{D}}=e^{z_0\frac{d}{dz}}Y(u,z)=Y(u, z+z_0).$
\end{enumerate}
\end{proposition}

\begin{definition}
An  {\em automorphism} of a vertex algebra $V$ is an invertible linear map $g$ of $V$ such that
$g\1=\1$ and $gY(u,z)g^{-1}=Y(gu,z)$ for all $u\in V$. We denote by $\text{Aut}(V)$ the set of all automorphisms of $V$.
\end{definition}

\begin{definition}
A {\em derivation} of a vertex algebra $V$ is a linear map $d$ of $V$
such that $d{\bf 1}=0$ and $[d, Y(v,z)]=Y(dv,z)$ for all $v \in V$.
Denote the set of all derivations of $V$ by $\text{Der}(V)$.
Note that $\text{Der}(V)$ is a Lie algebra.
\end{definition}

\begin{definition}
\begin{enumerate}[{(1)}]
\item A vertex algebra $V$ is said to be {\em irreducible} if $V$ is an irreducible $V$-module.

\item A vertex algebra $V$ is said to be {\em simple} if it does not contain any $\mathcal{D}$-stable $V$-submodule other than $0$ and $V$.
\end{enumerate}
\end{definition}

\begin{remark}
\begin{enumerate}[{(1)}]
\item If $I$ is a  $\mathcal{D}$-stable $V$-submodule if and only if $I$ is an ideal of $V$ defined in \cite{LL}.

\item If $V$ is a vertex operator algebra, then $V$ is irreducible if and only if $V$ is simple.

\item However, not every simple vertex algebra is irreducible.
For example, the commutative vertex algebra $(\C[x], \frac{d}{dx})$ is simple, but not irreducible.
\end{enumerate}
\end{remark}

\subsection{Hopf algebras}

 In this section, we review some basic facts on Hopf algebras. 
 Let $H=(H,\mu,\eta,\Delta,\epsilon,S)$ be a Hopf algebra,
 where the linear maps
 $$\mu:H\otimes H\rightarrow H,~\eta:\mathbb{C}\rightarrow H,~\Delta:H\rightarrow H\otimes H,~\epsilon:H\rightarrow \mathbb{C}, \ \
 \text{and} \ \ S:H\rightarrow H$$
 represent  multiplication, unit, comultiplication, counit, and antipode, respectively.
 Throughout this paper, we will use Sweedler's notation,
where for $h \in H$, we write $\Delta(h)=\sum h_1 \otimes h_2$ and  write
$$(\Delta \otimes Id)\Delta(h)=(Id \otimes \Delta)\Delta(h)=\sum h_1 \otimes h_2 \otimes h_3.$$



\begin{definition}
A Hopf algebra $H$ is called {\em cocommutative } if $\sum h_1 \otimes h_2 = \sum h_2 \otimes h_1$ for any $h \in H$.
\end{definition}

The following Lemma is well-known (see e.g. \cite{M}).

\begin{lemma} \label{gr-alg}
If $H$ is a finite-dimensional cocommutative Hopf algebra, then it is a group algebra.
\end{lemma}

\begin{definition}
A subspace $I$ of a Hopf algebra $H$ is called a {\em Hopf ideal} if it satisfies the following conditions:

\begin{enumerate}[{(1)}]
\item  $IH \subseteq I$ and $HI \subseteq I$.

\item  $\Delta(I) \subseteq H \otimes I + I \otimes H$ and $\epsilon(I) =0$.

\item  $S(I) \subseteq I$.
\end{enumerate}
A subspace $I$ of $H$ with properties (1) and (2) is called a {\em bialgebra ideal} of $H$.

\end{definition}

The following Lemma from \cite{N} is useful later.

\begin{lemma}  \label{bi-ideal}
 Every bialgebra ideal of a finite-dimensional Hopf algebra is a Hopf ideal.
\end{lemma}

We will also need the following facts. Let $M$ and $N$ be $H$-modules.
\begin{enumerate}[{(1)}]
\item $M \otimes N$ has an $H$-module structure under the action defined by
$$h(m \otimes n)=\sum h_1m \otimes h_2n$$
for $h \in H, m \in M$, and $n \in N$.

\item $\text{Hom}_{\C}(M, N)$ is an $H$-module under the action defined by
$$(h \cdot f)(m)=\sum h_1f(S(h_2)m)$$
for $h \in H, f \in \text{Hom}_{\C}(M, N), m \in M$.

\item Let $M^H=\{m \in M \ | \ hm=\epsilon(h)m \}.$ Then
$$\text{Hom}_{H}(M, N) = \text{Hom}_{\C}(M, N)^H.$$
\end{enumerate}

\begin{definition}
Given a Hopf algebra $H$ and an associative algebra $A$, we say that $H$ acts on $A$ (or that $A$ is an {\em $H$-module algebra}) if
the following conditions hold:
\begin{enumerate}[{(1)}]
\item $A$ is an $H$-module,

\item $h 1= \epsilon(h)1$ for any $h \in H$,

\item $h(ab)=\sum (h_1a)(h_2b)$ for any $h \in H,  a, b \in A$.
\end{enumerate}
\end{definition}

\begin{definition}
Given an $H$-module $M$, we say that $M$ is an {\em inner faithful} $H$-module if $IM \neq 0$ for every nonzero Hopf ideal $I$ of $H$.

Given an action of a Hopf algebra $H$ on an algebra $A$, we say that this action is inner faithful if $A$ is an inner faithful $H$-module.
\end{definition}

\begin{example} \label{Sweedler}
Sweedler's Hopf algebra $\mathcal{H}$ is the Hopf algebra generated by $g$ and $ x$ with relations
$$g^2=1, \ \ x^2=0, \ \ xg=-gx$$
and Hopf algebra structure
$$\Delta(g)=g \otimes g, \  \Delta (x) =x \otimes 1+ g\otimes x,  \ \epsilon(g)=1, \  \epsilon(x)=0,\ S(g)=g, \ \text{and} \ S(x)=-gx.$$
According to \cite{EW}, the polynomial algebra $\C[z]$ is an inner faithful $\mathcal{H}$-module algebra
under the action defined by
$$gz=-z, ~~~~xz=1.$$

\end{example}

\section{Hopf actions on vertex algebras}

\begin{definition} \cite{DW} \label{D}
Given a Hopf algebra $H$ and a vertex algebra $V$, we say that $H$ acts on $V$ (or that $V$ is an {\em $H$-module vertex algebra})
if the following conditions hold:
\begin{enumerate}[{(1)}]
  \item $V$ is an $H$-module.

  \item $h {\bf 1}=\epsilon(h){\bf 1}$, for any $ h\in H.$

  \item $h (Y(u,z)v)=\sum Y(h_1u,z)h_2v,$ for any $ h\in H,~u,v\in V.$
\end{enumerate}
\end{definition}

\begin{definition} \label{def-inner}
Given an action of a Hopf algebra $H$ on a vertex algebra $V$, we say that this action is {\em inner faithful} if $V$ is an inner faithful $H$-module.
\end{definition}

\begin{remark}
When discussing Hopf actions on vertex algebras (or algebras), one may argue that inner faithfulness is a more useful notion than faithfulness,
as one can pass uniquely to an inner faithful Hopf action if necessary:
Given an $H$-module vertex algebra $V$, it is easy to see that $H$ has a unique maximal Hopf ideal $I$ with $IV=0$.
Consequently, $H/I$ forms  a Hopf algebra and $V$ becomes an inner faithful $H/I$-module vertex algebra.
\end{remark}

\begin{definition}
Let $H$ be a Hopf algebra, and let $V$ be an $H$-module vertex algebra.
A {\em Hopf $V$-module} is a $V$-module $M$ that is also an $H$-module,  
satisfying the following compatibility condition for any $ h\in H,~ u \in V,$ and $w \in M$:  
\begin{align} \label{Hopf-V-module}
h (Y_M(u,z)w)=\sum Y_M(h_1u,z)h_2w.
\end{align}

\end{definition}

Let $V$ be an $H$-module vertex algebra. Define
$$V^H=\{v \in V \ | \ hv=\epsilon(h)v \}.$$

\begin{lemma} \label{le2-20}
Let $V$ be an $H$-module vertex algebra, and let $M$ be a Hopf $V$-module. Then

\begin{enumerate}[{(1)}]
  \item $V^H$ is a vertex subaglebra of $V$.

  \item The actions of $H$ and $\mathcal{D}$ on $V$ commute.

  \item The actions of $H$ and $V^H$ on $M$ commute.
\end{enumerate}
\end{lemma}
\begin{proof}
(1) See \cite[Lemma 3.6]{DW}.

(2) By definition, for $h \in H$ and $v \in V$, we have
$$h(\mathcal{D}v)=h(v_{-2}{\bf 1})=\sum (h_1v)_{-2}(h_2{\bf 1})=\sum (\epsilon(h_2)h_1v)_{-2}{\bf 1}=(hv)_{-2}{\bf 1}=\mathcal{D}(hv).$$
as required.

(3) For $v \in V^H,  h \in H,   w \in M$, and $n \in \Z $, we have
$$hv_nw=\sum(h_1v)_nh_2w=\sum (\epsilon(h_1)v)_nh_2w=\sum v_n(\epsilon(h_1)h_2)w=v_nhw,$$
as required.

\end{proof}

\begin{example}
For any Hopf algebra $H$ and any vertex algebra $V$, we have the trivial action $hv=\epsilon(h)v$ for all $h \in H, v \in V$.
\end{example}

\begin{example}
Let $V$ be a vertex algebra and $G$ an automorphism group of $V$.
Then the group algebra $\C[G]$ forms a Hopf algebra, and $V$  becomes a $\C[G]$-module vertex algebra.
We note that $$V^{\C[G]}=V^G:=\{v \in V \ | \ gv=v \ \text{for all} \ g \in G \}.$$

\end{example}

\begin{example} \label{exp-der-action}
Consider  $\mathfrak{g}$ as  a Lie subalgebra of $\text{Der}(V)$.
Let $U(\mathfrak{g})$ be the universal enveloping algebra of the Lie algebra $\mathfrak{g}$.
Then $U(\mathfrak{g})$ forms a Hopf algebra, and $V$ becomes a $U(\mathfrak{g})$-module vertex algebra.
We note that
$$V^{U(\mathfrak{g})}=V^\mathfrak{g}:= \{v \in V  \  | \ gv=0 \ \  \text{for all } \   g \in \mathfrak{g} \}.$$

\end{example}

\begin{example}
Let $V$ be a vertex algebra with an automorphism $\sigma$ of order two.
Let $\mathcal{H}$ be the Sweedler's Hopf algebra (see Example \ref{Sweedler}).
Then $V$ is an $H$-module vertex algebra such that $g$ acts on $V$ as $\sigma$, and $x$ acts trivially on $V$.
\end{example}

\begin{example} \label{ex2-22}
Let $H$ be a Hopf algebra. Let $(A, \partial)$ be a commutative differential algebra.
Assume that $A$ is an $H$-module algebra such that the actions of $H$ and  $\partial$ on $A$ commute with each other.
Then $(A, Y^{(A,\partial)}(z), {\bf 1})$ is an $H$-module vertex algebra.
\end{example}

At the end of this section, we provide two  important  examples  of inner faithful actions.

\begin{proposition} \label{ex-inn-fai1}
Let $V$ be a vertex algebra and let $G$ be an automorphism group of $V$.
Then $V$ is an inner faithful $\C[G]$-module vertex algebras.
\end{proposition}
\begin{proof}
Assume that the action is not inner faithful.
Then there exists a nonzero Hopf ideal $I$ of $H$ such that $IV=0$.
Let $f : \C[G] \to \C[G]/I$ be the canonical Hopf epimorphism.
We first claim that the restriction of $f$ to $G$, denoted by $f_{|G}$,  is not injective.

Assume that $f_{|G}$ is injective. Then for every $g \in G$, $f(g)$ is a nonzero group-like element in $\C[G]/I$.
Consequently, the set $\{f(g) \ | \ g \in G \}$ is linearly independent.
Now consider a nonzero element $\sum \lambda_i g_i \in I$, where $g_i \in G$, and $\lambda_i \in \C$ are all nonzero.
We have $ \sum \lambda_i f(g_i)=f(\sum \lambda_i g_i )=0 $, which is a contradiction.
Thus, $f_{|G}$ must be  non-injective.

Therefore, there exist distinct $g, h \in G$ such that $f(g)=f(h)$.
This implies that $g-h \in I$, and consequently $g=h$ on $V$.
This contradiction shows that  $H$ actions on $V$ must be  inner faithful. The proof is complete.
\end{proof}

\begin{remark}
If $V$ is $\pi_2$-injective and $G$ is a finite group,
we will further prove in Proposition \ref{G-faithful} that $V$ is a faithful $\C[G]$-module.
\end{remark}

\begin{proposition}
Let $V$ be a vertex algebra, and let $\mathfrak{g}$ be  a Lie subalgebra of $\text{Der}(V)$.
Then $V$ is an inner faithful  $U(\mathfrak{g})$-module vertex algebra.
\end{proposition}

\begin{proof}
Assume that the action is not inner faithful.
We can find a nonzero  Hopf ideal  $I$ of $U(\mathfrak{g})$ such that $IV=0$.
Let $f : U(\mathfrak{g}) \to U(\mathfrak{g})/I$ be the canonical Hopf epimorphism.
We claim that the restriction of $f$ to $\mathfrak{g}$, denoted by $f_{|\mathfrak{g}}$,  is not injective.

Suppose, for contradiction, that $f_{|\mathfrak{g}}$ is injective.
Then, $\overline{\mathfrak{g}}=f(\mathfrak{g})$ is a Lie algebra isomorphic to $\mathfrak{g}$.
Consequently, the Lie algebra isomorphism $f: \mathfrak{g} \to  \overline{\mathfrak{g}}$ can be extended to an  isomorphism of Hopf algebra
$F: U(\mathfrak{g}) \to   U(\overline{\mathfrak{g}})$.

On the other hand, since $U(\mathfrak{g})/I$  is generated by $\overline{\mathfrak{g}}$,
we can deduce from \cite{M} that
the inclusion of $\overline{\mathfrak{g}}$ into $U(\mathfrak{g})/I$ can be extended to an isomorphism of Hopf algebras
$\Phi: U(\overline{\mathfrak{g}}) \to U(\mathfrak{g})/I$.
However, we now have $\Phi F (I)=0$, which contradicts the fact that $F$ and $\Phi$ are isomorphic.


Thus, our assumption that  $f_{|\mathfrak{g}}$  is injective must be false,
and there exist  $g, h \in \mathfrak{g}$ with $g \neq h$ such that $g-h \in I$.
Consequently, $(g-h)V=0$, which implies that $g=h$ on $V$.
This contradiction shows that $ U(\mathfrak{g})$ actions on $V$  must be inner faithful. The proof is complete.

\end{proof}

\section{$\pi_2$-injective vertex algebras }

In this section, we introduce $\pi_2$-injective vertex algebras and provide numerous examples of such algebras.
We would like to emphasize that $\pi_2$-injective vertex algebras play a crucial role in the next section.

For $n \geq 1$, let
$$\pi^V_n: V^{\otimes n} \rightarrow  V((z_1)) \cdots ((z_{n-1})) $$
be the linear map defined by
$$\pi^V_n (v^1 \otimes \cdots \otimes v^{n-1} \otimes v^{n}  )=Y(v^1, z_1) \cdots Y(v^{n-1},z_{n-1})v^n$$
for $v^1, \cdots, v^n \in V.$

\begin{definition}
A vertex algebra $V$ is said to be {\em $\pi_2$-injective } if the map $\pi^V_2$ is injective.
\end{definition}

Before presenting examples of $\pi_2$-injective vertex algebras,
we first introduce a good property of such algebras.
Further properties of $\pi_2$-injective vertex algebras will be given in the next section.

\begin{proposition} \label{G-faithful}
Let $V$ be a vertex algebra such that $\pi^V_2$ is injective.
Let $G$ be a finite automorphism group of $V$.  Then
\begin{enumerate}[{(1)}]
\item The map $\pi^V_n$ is injective for any $n \geq 2$.

\item Every irreducible representation of $G$ appears in $V$.
\end{enumerate}
\end{proposition}
\begin{proof}
(1)  The proof is essentially  the same as that of  \cite[Lemma 5.1]{DRY}.
 We prove that $\pi^V_n$ is injective by induction on $n$.
Let $n \geq 2$ and assume that $v^1 \otimes \alpha_1+ \cdots + v^s \otimes \alpha_s \in \text{Ker}(\pi^V_n)$,
where $v^1, \cdots , v^s\in V$ are linearly independent and $\alpha_1, \cdots , \alpha_s \in V^{\otimes (n-1)}.$
Then, we have $$Y(v^1,z_1)\pi_{n-1}(\alpha_1)+ \cdots + Y(v^s, z_1) \pi_{n-1}(\alpha_s)=0.$$
Note from the definition that each $\pi_{n-1}(\alpha_i)\in V((z_2)) \cdots ((z_{n-1})).$ 
Since the map $\pi^V_2$ is injective, we deduce that $\pi_{n-1}(\alpha_1)= \cdots = \pi_{n-1}(\alpha_s)=0$
as the coefficients of each $z_2^{m_2}\cdots z_{n-1}^{m_{n-1}}$ in the equation above is zero for all $(m_2, \cdots ,m_{n-1})\in \Z^{n-2}.$
By induction, we conclude that  $\alpha_1= \cdots = \alpha_s=0$. Therefore, $\text{Ker}(\pi^V_n)=\{0\}$, as required.

(2) Since $\pi^V_n$ is injective for any $n \geq 1$,
the proof follows a similar argument as in \cite[Corollary 5.2]{DRY}.

\end{proof}

\begin{proposition} \label{pi2inj-exmples}
Let $V$ be a simple vertex algebra of countable dimension. Then the map $\pi^V_2$ is injective.
In particular, every irreducible vertex algebra of countable dimension is $\pi_2$-injective.

\end{proposition}
\begin{proof}

Suppose that the linear map $\pi^V_2$ is not injective.
Then there exists  a nonzero vector
$v^1 \otimes w^1 + \cdots v^s \otimes w^s$
in the kernel of $\pi^V_2$,
where $s$ is a positive integer, $v^1, \cdots, v^s$ are linearly independent vectors in $V$,
and $w^1 , \cdots , w^s$ are nonzero vectors in $V$.
Then we have $$Y(v^1, z)w^1+ \cdots + Y(v^s,z)w^s=0.$$
By weak associativity, for any $u \in V$, there exists some $k \in \N$ such that
\begin{align*}
& (z+z_0)^k(Y(Y(u, z_0)v^1, z)w^1 + \cdots + Y(Y(u, z_0)v^s, z)w^s)\\
&=(z_0+z)^k(Y(u, z_0+z)Y(v^1,z)w^1 + \cdots + Y(u, z_0+z)Y(v^s,z)w^s)=0,\\
\end{align*}
which implies that
\begin{align*}
Y(Y(u, z_0)v^1, z)w^1 + \cdots + Y(Y(u, z_0)v^s, z)w^s=0.
\end{align*}
On the other hand, we have
\begin{align*}
& Y(\mathcal{D} v^1, z)w^1+ \cdots + Y(\mathcal{D} v^s,z)w^s \\
& =\frac{d}{dz}(Y(v^1, z)w^1+ \cdots + Y(v^s,z)w^s)=0.
\end{align*}

To continue the proof, we consider the associative subalgebra $A(V, \mathcal{D})$  of $\text{End}(V)$,
which is generated by the operators $\mathcal{D}$ and $u_n$, where $u$ belongs to $V$ and $n$ belongs to $\Z$.
Therefore, for any $a \in A(V, \mathcal{D})$, we have
$$Y(av^1, z)w^1+ \cdots + Y(av^s,z)w^s=0.$$
Note that $V$ is a simple vertex algebra if and only if $V$ is an irreducible $A(V, \mathcal{D})$-module.
Hence by Jacobson density theorem there exists an $a \in R$ such that
$av^1={\bf 1}$ and $av^i=0$ for any $i \neq 1$.
It follows that $Y({\bf 1}, z)w^1=0$, which is a contradiction.
Hence $\pi^V_2$ is injective and the proof is complete.
\end{proof}

\begin{remark}
When $V$ is a simple vertex operator algebra,
the $\pi_2$-injectivity of $V$ directly follows from \cite[Lemma 3.1]{DM}.
Similarly, when $V$ is an irreducible vertex algebra of countable dimension, the $\pi_2$-injectivity of $V$ can be found in \cite{DRY,Li1}.

\end{remark}


\begin{proposition} \label{pi-ex1}
For any non-negative integer $m$, the commutative vertex algebra
$$(V=\C[x], \ \partial=x^m\frac{d}{dx})$$ is $\pi_2$-injective.
\end{proposition}

\begin{proof}

Let $$f^1 \otimes g^1 + \cdots + f^s \otimes g^s \in \text{Ker}(\pi^V_2),$$
where $s$ is a positive integer and $f^1, g^1, \cdots, f^s, g^s \in \C[x].$
From the definition of $\pi^V_2$ and $Y^{(V, \partial)}(z)$, we have
\begin{align} \label{eq3-1}
(e^{z\partial }f^1)g^1+ \cdots +(e^{z\partial }f^s)g^s=0.
\end{align}
By comparing the coefficients of $z^k$ on both sides of equation (\ref{eq3-1}), we obtain
\begin{align} \label{eq3-2}
(\mathcal{\partial}^kf^1)g^1+ \cdots + (\mathcal{\partial}^kf^s)g^s=0
\end{align}
for any $k \geq 0$.

For each $i$, we can write $f^i=x^{n_i}+\widetilde{f^i}$ and $g^i= \lambda_i x^{m_i} + \widetilde{g^i}$,
where $\lambda_i$ is a non-zero complex number, $n_i, m_i$ are non-negative integers,
and $\widetilde{f^i}$ (resp. $\widetilde{g^i}$) is the sum of monomials in $f^i$ (resp. $g^i$)
whose degree is less than $n_i$ (resp. $m_i$).
For any $k \geq 0$, by comparing the highest degree terms of
$(\partial^kf^1)g^1,  \cdots, (\partial^kf^s)g^s $, we obtain
$$x^{n_1} \otimes \lambda_1 x^{m_1} + \cdots + x^{n_s} \otimes \lambda_s x^{m_s}  \in \text{Ker}(\pi^V_2).$$

We claim that $x^{n_1} \otimes \lambda_1 x^{m_1} + \cdots + x^{n_s} \otimes \lambda_s x^{m_s}=0.$
Without loss of generality, we can assume that
$s \geq2$, $n_1+m_1 = \cdots =n_s+m_s$, and $n_1>n_2> \cdots > n_s \geq 0$.
By Equation (\ref{eq3-2}), we have
\begin{align*}
& \lambda_1+\lambda_2+ \cdots +\lambda_s=0, \\
& n_1 \lambda_1 +n_2 \lambda_2 + \cdots + n_s \lambda_s=0,\\
& [n_1, 1]\lambda_1 + [n_2,1] \lambda_2 + \cdots + [n_s,1] \lambda_s=0,\\
& \cdots   ~~~~~~~~~~~~~~~~~~~~ \cdots ~~~~~~~~~~~~~~~~~~~ \cdots \\
&[n_1, s-2]\lambda_1 + [n_2,s-2] \lambda_2 + \cdots + [n_s,s-2] \lambda_s=0,
\end{align*}
where $[n_i, s]$  represents the product $n_i(n_i+m-1) \cdots (n_i+s(m-1))$.
A simple calculation shows that
\[
\text{det}
\begin{pmatrix}
1 & 1 & \cdots & 1 \\
n_1 & n_2 & \cdots & n_s \\
\vdots & \vdots & \ddots & \vdots \\
[n_1, s-2] & [n_2,s-2] & \cdots & [n_s, s-2]
\end{pmatrix}
= \text{det}
\begin{pmatrix}
1 & 1 & \cdots & 1 \\
n_1 & n_2 & \cdots & n_s \\
\vdots & \vdots & \ddots & \vdots \\
n_1^{s-1} & n_2^{s-1} & \cdots & n_s^{s-1}
\end{pmatrix}
\]
From the Vandermonde determinant, we see that $\lambda_1 =\lambda_2 = \cdots =\lambda_s=0$, as required.

Now, we have
$$\sum_{i=1}^s ( x^{n_i} \otimes \widetilde{g^i} + \widetilde{f^i} \otimes x^{m_i} + \widetilde{f^i} \otimes \widetilde{g^i})
=f^1 \otimes g^1 + \cdots + f^s \otimes g^s \in \text{Ker}(\pi^V_2).$$
By repeating the above process, we can conclude that
$$\sum_{i=1}^s ( x^{n_i} \otimes \widetilde{g^i} + \widetilde{f^i} \otimes x^{m_i} + \widetilde{f^i} \otimes \widetilde{g^i})=0.$$
Consequently, $f^1 \otimes g^1 + \cdots + f^s \otimes g^s=0$.
Hence $\text{Ker}(\pi^V_2)=0$, and the proof is complete.

\end{proof}

\begin{remark}
If $m >0$, then $<x^m>$ is a differential ideal of $(\C[x], x^m \frac{d}{dx})$.
Hence for any $m > 0$, the vertex algebra $(\C[x], x^m \frac{d}{dx})$ is not simple.

\end{remark}

In order to obtain more $\pi_2$-injective vertex algebras, we review the definition of non-degenerate vertex algebras introduced in \cite{EK,Li1}.
For $n \geq 1$, let
$$Z^V_n: V^{\otimes n} \otimes \C((z_1))\cdots((z_n)) \to V((z_1))\cdots((z_n)) $$
be the linear map defined by
$$Z^V_n(v^1 \otimes \cdots v^n \otimes f)=fY(v^1, z_1)\cdots Y(v^n,z_n){\bf 1}$$
for $v_1, \cdots, v_n \in V$, and $f \in \C((z_1))\cdots((z_n))$.

\begin{definition}
A vertex algebra $V$ is said to be {\em nondegenerate} if for any $n \geq 0$, the linear map $Z^V_n$ is injective.
\end{definition}

\begin{remark}
It is shown in \cite{Li2} that if $V$ is an irreducible vertex algebra of countable dimension, then $V$ is nondegenerate.
 Additionally, \cite{Li2} also points out that a simple vertex algebra is not necessarily nondegenerate.
\end{remark}

\begin{proposition} \cite{Li1}
If $V$ is a nondegenerate vertex algebra, then $V$ is $\pi_2$-injective.
\end{proposition}
\begin{proof}
Here we give a direct proof.
Assume that $\pi^V_2$ is not injective.
Let $0 \neq \sum_{i=0}^s v^i \otimes u^i \in \text{Ker}(\pi^V_2)$ for some $s$.
Then $ \sum_{i=0}^s Y(v^i,z)u^i=0.$
By Proposition \ref{ED-der}, we have
\begin{align*}
0 &=\sum_{i=0}^s Y(v^i,z_1-z_2)u^i=\sum_{i=0}^s e^{-z_2 \mathcal{D}} Y(v^i,z_1)e^{z_2 \mathcal{D}}  u^i\\
&= \sum_{i=0}^s e^{-z_2 \mathcal{D}} Y(v^i,z_1)Y(u^i,z_2){\bf 1}.
\end{align*}
Consequently, we have
\begin{align*}
Z^V_2(\sum_{i=0}^s v^i \otimes u^i \otimes 1)=\sum_{i=0}^s Y(v^i,z_1)Y(u^i,z_2){\bf 1}=0.
\end{align*}
This contradicts the fact that the map $Z^V_2$ is injective. The proof is complete.
\end{proof}

\begin{proposition}\cite{Li2}
Let $U$ and $V$ be nondegenerate vertex algebras. Then the tensor product vertex algebra $U \otimes V$ is also nondegenerate.
In particular,   $U \otimes V$ is $\pi_2$-injective.
\end{proposition}

\begin{remark}
The tensor product vertex algebra of  two $\pi_2$-injective vertex algebras  may not be $\pi_2$-injective:
Let $U=(\C[x], \frac{d}{dx})$ and $V=(\C[y], \frac{d}{dy})$.
Then $U$ and $V$ are $\pi_2$-injective vertex algebras by Lemma \ref{pi-ex1}.
However, the tensor product vertex algebra
$U \otimes V \cong (\C[x,y], \frac{\partial}{\partial x}+\frac{\partial}{ \partial y})$ is not $\pi_2$-injective.
For example, it is easy to see that $(x-y)\otimes 1- 1 \otimes (x-y) \in \text{Ker} (\pi_2).$
\end{remark}

\begin{example} \cite{Li2} \label{ex-FCD}
Let $U$ be a vector space. Set
$$FCD(U)=S(U \otimes \C[\partial]),$$
the symmetric algebra over the vector space $U \otimes \C[\partial]$, where $\partial$ is a formal variable.
Note that $FCD(U)$ is  the free commutative differential algebra over $U$ with the derivation $\partial$.

If $U$ is a vector space of countable dimension,
then the free commutative differential algebra $(FCD(U), \partial)$, viewed as a vertex algebra, is nondegenerate.
In particular, the vertex algebra $(FCD(U), \partial)$ is $\pi_2$-injective.
\end{example}

\begin{example} \cite{Li2}
Let $\mathfrak{g}$ be any Lie algebra of countable dimension
equipped with a nondegenerate symmetric invariant bilinear form.
For $k \in \C$, let $V_{\mathfrak{\hat{g}}}(k,0)$ be the level $k$ vacuum module
vertex algebra associated with the Lie algebra $\mathfrak{g}$ (see \cite{LL}).
Then $V_{\mathfrak{\hat{g}}}(k,0)$  is a nondegenerate vertex algebra.
In particular, the vertex algebra $V_{\mathfrak{\hat{g}}}(k,0)$ is $\pi_2$-injective.
\end{example}

\begin{example} \cite{Li2}
Let $V_{vir}(c, 0)$ be the (universal) Virasoro vertex operator algebra with the central charge $c \in \C$ (see \cite{LL}).
Then the vertex operator algebra $V_{vir}(c, 0)$ is  nondegenerate.
In particular, the vertex algebra $V_{vir}(c, 0)$ is $\pi_2$-injective.

\end{example}

\section{Hopf actions on $\pi_2$-injective vertex algebras }

The goal of this section is to prove Theorem \ref{thm-gr-alg} and Theorem \ref{thm-Hopf-ideal} below. To do this, we need some lemmas.
\begin{lemma} \label{linear-inj}
Let $M$ and $N$ be two vector spaces (not necessarily finite-dimensional).
 Let $$\lambda_{M,N}: \text{End}_\C(M) \otimes \text{End}_\C(N) \to \text{End}_\C(M \otimes N)$$ be
 the linear map defined by $$\lambda_{M,N}(f \otimes g)(m \otimes n)=f(m) \otimes g(n)$$
for $f \in  \text{End}_\C(M), g \in \text{End}_\C(N), m \in M $, and $n \in N$.
Then the linear map $\lambda_{M,N}$ is injective.
\end{lemma}

The following statement will be useful later.  Consider two associative algebras $A$ and $B$.
Let  $M$ be an $A$-module and $N$ a $B$-module.
We note that $M \otimes N$ is an $A \otimes B$-module,
where the action is defined by $(a \otimes b)(m \otimes n)= am \otimes bn$ for $a \in A, b \in B, m \in M $, and $n \in N$.
Furthermore, if $M$ is a faithful $A$-module and $N$ is a faithful $B$-module,
then by Lemma \ref{linear-inj}, $M \otimes N$ is a faithful $A \otimes B$-module.

\medskip

For a positive integer $n$, let $S_n$ be the symmetric group on the set $\{1, 2, \cdots, n \}$.
For $\tau \in S_n$, we define a linear map $\widetilde{\tau}: V^{\otimes n} \to V^{\otimes n}$ by
$$\widetilde{\tau }(v_1 \otimes \cdots v_n)=v_{\tau 1} \otimes \cdots v_{\tau n}$$
for $v_1, \cdots , v_n \in V$.

\begin{lemma} \label{H-iso}
Let $V$ be a vertex algebra such that $\pi^V_2$ is injective, and let $\tau \in S_n$.
Then the linear map $\widetilde{\tau}$ defined above is an $H$-isomorphism.
\end{lemma}

\begin{proof}
To prove that the linear map $\widetilde{\tau}$ is an $H$-isomorphism,
it is sufficient to show that the linear map  $\varphi: V \otimes V \to V \otimes V $ defined by
$\varphi (u \otimes v) =v \otimes u$ for $u, v \in V$,  is an $H$-isomorphism.
To continue the proof, we set
$$\mathcal{V}^0=\text{span}\{Y(u,z)v \ | \ u, v \in V \}\subseteq V\{z\},$$
and $$\mathcal{V}^1=\text{span}\{Y(u,-z)v \ | \ u, v \in V \}\subseteq V\{z\}.$$
We note that $V\{z\}$ has an $H$-module structure under the action defined by
$$h(\sum v^i z^i)=\sum (hv^i)z^i$$ for $h \in H, v^i \in V$.
As $V$ is an $H$-module vertex algebra, we can see that $\mathcal{V}^0$ and $\mathcal{V}^1$ are $H$-submodules of $V\{z\}$.
We note that $e^{-z\mathcal{D}} \mathcal{V}^0=\mathcal{V}^1.$
Furthermore, since the actions of $\mathcal{D}$ and $H$ on $V$ commute (see Lemma \ref{le2-20}(2)),
the map $e^{-z\mathcal{D}}: \mathcal{V}^0 \rightarrow \mathcal{V}^1$ is an $H$-isomorphism.

Since $\pi^V_2$ is injective,  the map $\pi^V_2: V \otimes V \to \mathcal{V}^0$ is an $H$-isomorphism.
Similarly, it is easy to verify that the linear map $\widetilde{\pi}^V_2: V \otimes V \to \mathcal{V}^1$ defined
by $\widetilde{\pi}^V_2 (u \otimes v)=Y(u,-z)v$ for $u, v \in V$ is also an $H$-isomorphism.
A straightforward calculation shows that $\varphi =(\widetilde{\pi}^V_2)^{-1}e^{-z\mathcal{D}}\pi^V_2$.
Therefore $\varphi$ is an $H$-isomorphism.
The proof is complete.

\end{proof}

\begin{remark}
When $V$ is a simple vertex operator algebra, it is shown in \cite{DW} that $\varphi$ is an $H$-isomorphism.
\end{remark}

\begin{lemma} \label{inner-to-faith}
Let $H$ be a finite-dimensional Hopf algebra.
Assume that $V$ is an inner faithful $H$-module vertex algebra.
Then there exists some $s_0$ such that $H$ acts faithfully on $V^{\otimes s}$  for any $s \geq s_0$.
\end{lemma}
\begin{proof}
To prove it, we will use a similar argument as presented in \cite{EW}.
For $s \geq 1$, let $K_s \subset H$ be the kernel of the action of $H$ on $V^{\otimes s}$.
Since $V^{\otimes s} \cong V^{\otimes s} \otimes {\bf 1} \subseteq V^{\otimes (s+1)}$, we see that $K_{s+1} \subseteq K_s$.
Let $K=\bigcap_{i\geq 0}K_i$. Since $H$ is finite-dimensional, there is an integer $s_0$ such that $K=K_s$ for all $s \geq s_0$.
To prove that $V^{\otimes s}$  for any $s \geq s_0$ is a faithful $H$-module, we must show that $K=0$.

Let $k \in K$. By the definition of $K$, the action of $k$ on
$V^{\otimes (s+t)}=V^{\otimes s} \otimes V^{\otimes t}$  is zero for any $s ,t \geq 1$.
We note  $V^{\otimes s} \otimes V^{\otimes t}$ for any $s, t \geq s_0$ is a faithful module over $H/K \otimes H/K$.
Therefore, we have  $\Delta(k) \in K \otimes H +H \otimes K$.
On the other hand, since $k V=0$, we can  deduce that $\epsilon(k)=0$.
Consequently, $K$ is a bialgebra ideal of $H$, and hence a Hopf ideal by Lemma \ref{bi-ideal}.
Since the action of $H$ on $V$ is inner faithful and $KV=0$, hence we have $K=0$, as required.
\end{proof}


\begin{theorem} \label{thm-gr-alg}
Let $H$ be a finite-dimensional Hopf algebra.
Let $V$ be a vertex algebra such that $\pi_2^V$ is injective.
Assume that $V$ is an inner faithful $H$-module vertex algebra.
Then $H$ is a group algebra.
\end{theorem}

\begin{proof}
According to Lemma \ref{inner-to-faith}, we can choose a positive integer $s$ such that $H$ acts faithfully on $V^{\otimes s}$.
According to Lemma \ref{H-iso},
the linear map $f: V^{\otimes s} \otimes V^{\otimes s} \to V^{\otimes s} \otimes V^{\otimes s}$ defined
by $f(\alpha \otimes \beta)=\beta \otimes \alpha$ for $\alpha, \beta \in V^{\otimes s}$, is an $H$-isomorphism.
As a consequence, we have
$$\sum h_1 \beta \otimes h_2 \alpha= \sum h_2 \beta \otimes h_1 \alpha$$
for any $\alpha, \beta \in V^{\otimes s}$, and  $h \in H$.
Since $V^{\otimes s} \otimes V^{\otimes s}$ is  a faithful $H \otimes H$-module,
we can conclude that $\Sigma  h_1 \otimes h_2 =\Sigma h_2 \otimes h_1$ for any $h \in H$.
Thus, $H$ is cocommutative and $H$ is a group algebra by Lemma \ref{gr-alg}.
The proof is complete.
\end{proof}

\begin{corollary}
Consider a vector space $U$ of countable dimension,
and let $FCD(U)=S(U \otimes \C[\partial])$ be the free commutative differential algebra over $U$
with the derivation $\partial$ (Example \ref{ex-FCD}).
Let $H$ be a finite-dimensional Hopf algebra.
Assume that  $FCD(U)$ is an inner faithful $H$-module algebra such that the actions of $H$ and  $\partial$ on  $FCD(U)$ commute.
Then $H$ is a group algebra.
\end{corollary}
\begin{proof}
Together with Example \ref{ex2-22}, Example \ref{ex-FCD} and Theorem \ref{thm-gr-alg}, the result follows.
\end{proof}

The following Corollary follows from  Lemma \ref{pi-ex1} and Theorem \ref{thm-gr-alg} immediately.

\begin{corollary} \label{mod-alg}
Let $H$ be a finite-dimensional Hopf algebra. Assume that
the polynomial algebra $\C[z]$ is an inner faithful $H$-module algebra such that
the actions of $H$ and  $x^m\frac{d}{dx}$, for some $m \geq 0$, on $\C[x]$ commute with each other.
Then $H$ is a group algebra.
\end{corollary}

\begin{remark}
It should be noted that the conclusion of Corollary \ref{mod-alg} does not hold
without additional assumptions that the actions of $H$ and  $x^m\frac{d}{dx}$, for some $m \geq 0$, on $\C[x]$ commute with each other.
For instance, in Example \ref{Sweedler}, the algebra $\C[z]$ serves as an inner faithful $\mathcal{H}$-module algebra,
where $\mathcal{H}$ is the Sweedler's Hopf algebra.
However, the Sweedler's Hopf algebra  $\mathcal{H}$ is not a group algebra.


\end{remark}

In \cite{DW}, the authors  propose a conjecture that the kernel of a Hopf action on a simple vertex operator algebra is a Hopf ideal.
In the following Theorem, we will prove that if the Hopf algebra involved in the conjecture is finite-dimensional,
then the conjecture is indeed valid.

\begin{theorem} \label{thm-Hopf-ideal}
Let $H$ be a Hopf algebra, and let $V$ be an $H$-module vertex algebra such that $\pi^V_2$ is injective.
Let $K$ be the kernel of the action of $H$ on $V$.
Then $K$ is a bialgebra ideal of $H$. In particular, if $H$ is finite-dimensional, then $K$ is a Hopf ideal of $H$.
\end{theorem}
\begin{proof}

It is easy to see that  $K$ is a two-sided ideal of $H$ and $\epsilon(K)=0$.
To prove that $K$ is a bialgebra ideal of $H$, it only remains to show that
$\Delta(K) \subseteq K \otimes H +H \otimes K.$
Assume that there exists an element $k \in K$ such that $\Delta(k)  \notin K \otimes H +H \otimes K.$
Consider a subspace $L$ of $H$ such that $H=K \oplus L$.
Then we can write
$$\Delta(k)=\sum_{i=1}^s a_i \otimes b_i +\sum_{j=1}^t c_j \otimes d_j$$
where, $s ,t$ are  non-negative integer,
$\sum_{i=1}^s a_i \otimes b_i \in  K \otimes H +H \otimes K,$
and $\sum_{j=1}^t c_j \otimes d_j \in L \otimes L$ is nonzero.
 Let $\rho: H \to \text{End}_\C(V)$ be the map defined by $\rho(h)(v)=hv$ for $h \in H$ and $v \in V$.
 Then $\rho$ is injective on $L$.
By Lemma \ref{linear-inj}, the map
$$\lambda_{V,V}: \text{End}_\C(V) \otimes \text{End}_\C(V) \to \text{End}_\C(V \otimes V)$$ is injective.
Hence, we have
 $$0 \neq \lambda_{V,V}(\rho \otimes \rho)(\sum_{j=1}^t c_j \otimes d_j) \in \text{End}_\C(V \otimes V).$$
Consequently, there exist elements $u, v \in V$ such that
$$0 \neq (\lambda_{V,V} (\rho \otimes \rho)(\sum_{j=1}^t c_j \otimes d_j))(u \otimes v)=\sum_{j=1}^t c_ju \otimes d_jv.$$
Since $\pi^V_2$ is injective, we have
$$kY(u, z)v=\sum_{i=1}^sY(a_iu, z)b_iv + \sum_{j=1}^tY(c_ju, z)d_jv= \sum_{j=1}^tY(c_ju, z)d_jv \neq 0,$$
which contradicts the fact that $k \in K$. Therefore, $K$ is a bialgebra ideal of $H$.
The result follows from Lemma \ref{gr-alg}.
\end{proof}

As an application of Theorem \ref{thm-Hopf-ideal},
we can provide a new proof for Theorem \ref{thm-gr-alg} that does not rely on Lemma \ref{inner-to-faith}.

{\bf The second proof of Theorem \ref{thm-gr-alg}}:
\ Let $K$ be the kernel of the action of $H$ on $V$.
By Theorem \ref{thm-Hopf-ideal}, $K$ is a Hopf ideal of $H$.
Since $V$ is an inner faithful $H$-module, we can conclude that $K=0$, which means that $V$ is a faithful $H$-module.
Therefore, by Lemma \ref{linear-inj}, the tensor product $V \otimes V$ is a faithful $H \otimes H$-module.
On the other hand, it follows from Lemma \ref{H-iso} that the linear map $\varphi: V \otimes V \to V \otimes V$ defined by
$\varphi(u \otimes v)=v\otimes u$ for $u, v \in V$ is an $H$-isomorphic.
Hence, we have $\Sigma h_1v \otimes h_2u = \Sigma h_2 v \otimes h_1 u$ for any $h \in H$ and $u, v \in V$.
Therefore, $\Sigma h_1 \otimes h_2 = \Sigma h_2 \otimes h_1$ for any $h \in H$.
Now the result follows from Lemma \ref{gr-alg}.


\section{Semisimple Hopf actions }

The goal of this section is to establish  a Schur-Weyl type duality for semisimple Hopf actions on Hopf modules of vertex algebras. 
We assume that $V$ is an irreducible vertex algebra of countable dimension
and $H$ is a Hopf algebra such that $V$ is an $H$-module vertex algebra.
Let $M$ be a Hopf $V$-module such that $M$ is an irreducible $V$-module.
Additionally, we assume that both $V$ and $M$ are direct sums of finite-dimensional irreducible $H$-modules.

Let $\Lambda$ be the set of  characters of all finite-dimensional irreducible representations of $H$.
For $\lambda \in \Lambda$, we denote the corresponding irreducible representation by $W_\lambda$.
Let $M^{\lambda}$  be the sum of all $H$-submodules of $M$  isomorphic to $W_{\lambda}$.
Let $M_{\lambda}=\text{Hom}_H(W_{\lambda}, M)$ be the multiplicity space of $W_{\lambda}$ in $M$.
Since the actions of $V^H$ and $H$ on $M$ commute (see Lemma \ref{le2-20}),
$M_{\lambda}$ has a $V^H$-module structure by defining
$$(v_nf)(w)=v_n(f(w))$$
for $v \in V^H, f\in M_{\lambda}, n \in \Z$,  and  $ w \in W_{\lambda}.$
Note that for $\lambda \in \Lambda$, the linear map
$$\theta^\lambda: W_{\lambda} \otimes M_{\lambda} \to V^{\lambda}$$
defined by $$\theta^\lambda(w \otimes f)=f(w)$$
for $w \in W_{\lambda}$ and $f \in M_{\lambda}$, is an $H \otimes V^H$-isomorphism (see \cite{DRY}).
Now,we can decompose $M$ as an $H \otimes V^H$-module as follows:
$$M=\bigoplus_{\lambda \in \Lambda}W_\lambda \otimes M_\lambda.$$

\begin{theorem} \label{main6-1}
The $H \otimes V^H$-module decomposition:
$$M=\bigoplus_{\lambda \in \Lambda}W_\lambda \otimes M_\lambda$$
gives a dual pair  $(H, V^H)$ on $M$ in the following sense:
\begin{enumerate}[{(1)}]

\item For every $\lambda \in \Lambda$ such that $M_\lambda \neq 0$,  $M_\lambda$ is an irreducible $V^H$-module.

\item If $M_\lambda$ and $M_\mu$ are nonzero, then $M_\lambda$ and $M_\mu$ are isomorphic $V^H$-modules if and only if $\lambda=\mu$.
\end{enumerate}
\end{theorem}

To proof Theorem \ref{main6-1}, we first note that $V \otimes \C[t, t^{-1}]$ has an $H$-module structure by defining
$$h(u \otimes t^n)=hu \otimes t^n$$
for $h \in H, u \in V$ and $n \in \Z$.
Moreover, we have
$$(V \otimes \C[t, t^{-1}])^H=V^H \otimes \C[t, t^{-1}].$$
For a subspace $X$ of $M$, define a linear map
\begin{align} \label{eq6-1}
\rho : V \otimes \C[t, t^{-1}] \rightarrow \text{Hom}_{\C}(X, M)
\end{align}
by $$\rho(u \otimes t^n)x=u_nx$$
for $u \in V, n \in \Z$ and $x \in X$.

\begin{lemma} \label{f=un} \cite{DRY}
Let $V$ be an irreducible  vertex algebra of countable dimension, and let $M$ be an irreducible $V$-module.
Let $X$ be a finite-dimensional subspace of $M$.
Then for any $f\in \text{Hom}_{\C}(X, M)$,
there exist $v^1, \cdots, v^n \in V$, and $i_1, \cdots, i_n \in \Z$ such that
$$f=v^1_{i_1}+ \cdots + v^n_{i_n}.$$
\end{lemma}

\begin{lemma} \label{f=ov}
If $X$ is a finite-dimensional $H$-submodule of $M$, then the map $\rho$ in (\ref{eq6-1}) is an $H$-epimorphism. In particular, we have
$$ \text{Hom}_{H}(X, M)= \text{Hom}_{\C}(X, M)^H= \rho(V^H \otimes \C[t, t^{-1}]).$$
\end{lemma}

\begin{proof}
It follows from Lemma \ref{f=un} that the linear map $\rho$ is surjective.
To complete the proof, we must show that $\rho$ is an $H$-homomorphism.
For $h \in H, u \in V, x \in X$, and $n \in \Z$, by identity (\ref{Hopf-V-module}),  we have
\begin{align*}
&(h \cdot \rho(u \otimes t^n))(x)=\sum h_1u_n(S(h_2)x)\\
&=\sum (h_1u)_n(h_2S(h_3)x)\\
&=\sum (h_1u)_n(\epsilon(h_2)x)\\
&=\sum (hu)_nx\\
&=\rho(h(u \otimes t^n))(x).
\end{align*}
Hence $\rho$ is an $H$-homomorphism.
Because $V$ is a semisimple H-module, the result follows.
\end{proof}

Now the proof of Theorem \ref{main6-1} is similar to that of \cite[Theorem 4.7]{DRY}.

\medskip

{\bf Proof of Theorem \ref{main6-1} }
(1) Let $\lambda \in \Lambda$ such that $M^\lambda \neq 0$.
For any nonzero elements $x$ and $y$ in $M_\lambda$,
we need to find  $v^1, \cdots, v^n \in V^H$, and $i_1, \cdots, i_n \in \Z$ such that $y =(v^1_{i_1}+ \cdots + v^n_{i_n})x$.
Consequently, $M_\lambda$ is an irreducible $V^H$-module.

Let $X= W_\lambda \otimes x$. Then $X$ is a finite dimensional $H$-submodule of $M$.
Define a  linear map $f \in \text{Hom}_\C(X, M)$ by $f(w \otimes x)=w \otimes y$  for $w \in W_\lambda$.
It is claer that $f \in \text{Hom}_H(X, M)$.
By Lemma \ref{f=ov}, there exist $v^1, \cdots, v^n \in V^H$, and $i_1, \cdots, i_n \in \Z$ such that
$$f=v^1_{i_1}+ \cdots + v^n_{i_n}.$$
Thus  we have
$$w \otimes y =f(w \otimes x)=(v^1_{i_1}+ \cdots + v^n_{i_n})(w \otimes x)=w \otimes (v^1_{i_1}+ \cdots + v^n_{i_n})x,$$
for any $w \in W_\lambda$. This implies that $(v^1_{i_1}+ \cdots + v^n_{i_n})x=y$, as required.

(2) Let $\lambda,  \mu \in \Lambda$ such that $\lambda \neq \mu$, and  $M_\lambda, M_\mu$ are nonzero.
Assume that $\phi: M_\lambda \to M_\mu$ is  a $V^H$-isomorphism.
Let $0\ne x\in M_\lambda$, and let $y=\phi(x)\in M_\mu$.
Set $$X= (W_\lambda \otimes x) \oplus (W_\mu \otimes y)$$  which is a finite-dimensional  $H$-submodule of $M$.
Define an $f \in  \text{Hom}_H(X, M)$
such that $$f(w_1\otimes x+w_2\otimes y)=w_2\otimes y$$
for $w_1\in W_\lambda$ and $w_2\in W_\mu.$
It follows from Lemma \ref{f=ov} that there exist $v^1, \cdots, v^n \in V$ and $i_1, \cdots, i_n \in \Z$ such that
$$f=v^1_{i_1}+ \cdots + v^n_{i_n}.$$
Then $(v^1_{i_1}+ \cdots + v^n_{i_n})x=0$ and $(v^1_{i_1}+ \cdots + v^n_{i_n})y=y.$
So
$$0=\phi((v^1_{i_1}+ \cdots + v^n_{i_n})x)=(v^1_{i_1}+ \cdots + v^n_{i_n})\phi(x)=(v^1_{i_1}+ \cdots + v^n_{i_n})y=y,$$
which is a contradiction. The proof is complete.      ~~~~~~~~~~~~~~~~~~~~~~~~~~~~~~~~~~~~~~~~~~~~~~~~~~~~~~~~~~~~~~~~~~~~$\Box$

\medskip

At the end of the article, we  provide two applications of Theorem \ref{main6-1} in $\N$-graded vertex algebra.
\medskip

\begin{definition}
An {\em $\N$-graded vertex algebra } is a vertex algebra $V$ equipped with a decomposition
 $V=\bigoplus_{n \in \N }V_n$ satisfying

\begin{enumerate}[{(1)}]
\item ${\bf 1} \in V_0$;

\item $v_sV_n \subseteq V_{n+m-s-1}$ for $v \in V_m, s \in \Z, m, n \in \N$.
\end{enumerate}
\end{definition}

\begin{definition}
Let $V$ be an $\N$-graded vertex algebra. An {\em admissible $V$-module} is an $\N$-graded $V$-module $M=\bigoplus_{n \in \N}M(n)$
satisfying $$v_sM(n) \subseteq M(n+m-s-1)$$ for $v \in V_m, s \in \Z, n, m \in \N$.
\end{definition}

{\bf Application I.} Let $V = \oplus_{n \in \N}V_n$ be an irreducible  $\N$-graded  vertex algebra
such that  $\text{dim}V_n < \infty$ for any $n \in \N$.
Let $G$ be a compact automorphism group of $V$.
Assume that the action of $G$ on $V$ is continuous, and that $g V_n \subseteq V_n$ for  any $g \in G, n \in \N$.
Let $\Lambda$ be the set of  characters of all finite-dimensional irreducible representations of $G$.
For $\lambda \in \Lambda$, we denote the corresponding irreducible representation by $W_\lambda$.
Let $V_{\lambda}=\text{Hom}_G(W_{\lambda}, V)$ be the multiplicity space of $W_{\lambda}$ in $V$.

Since each homogeneous subspace $V_n$ of $V$ is finite-dimensional,
hence $V$ is a direct sum of  finite-dimensional irreducible $G$-modules.
Now, the following Corollary is an immediate consequence of Theorem \ref{main6-1}.

\begin{corollary} \cite{DLM1} We have the following statements.
\begin{enumerate}[{(1)}]


\item For every $\lambda \in \Lambda$ with $V_\lambda \neq 0$,  $V_\lambda$  is an irreducible $V^G$-module;

\item If $V_\lambda$ and $V_\mu$ are  both nonzero, then $V_\lambda$ and $V_\mu$ are isomorphic $V^G$-modules if and only if $\lambda=\mu$.
\end{enumerate}
\end{corollary}

{\bf Application II.} Let $V = \bigoplus_{n \in \N}V_n$ be an irreducible $\N$-graded vertex algebra
such that $V_0= \C{\bf 1}$ and  $\text{dim}V_n < \infty$ for any $n \in \N$.
Let $M=\bigoplus_{n \in \N}M(n)$ be an simple admissible $V$-module such that $\text{dim}(M(n)) < \infty$ for any $n \in \N$.
Note that $V_1$ is a Lie algebra with Lie bracket $[u, v]=u_0v$ for $u, v \in V_1$.
Moreover, $M$ is a $V_1$-module under the action defined by $u\cdot m =u_0 m$  for $u \in V $ and $ m \in M$.
We also note that $u \cdot M(n) \subseteq M(n)$ for any $u \in V_1$ and $n \in \N$.

Let $\mathfrak{g}$ be a finite-dimensional semisimple Lie subalgebra of $V_1$.
Then both $V$ and $M$ are direct sums of finite-dimensional irreducible $\mathfrak{g}$-submodules.
Let $P$ be the set of all dominant weights of $\mathfrak{g}$.
For $\lambda \in P$, let $L(\lambda)$ be the irreducible highest weight module corresponding to $\lambda$. 
Let $M^{\lambda}$ be the sum of all $\mathfrak{g}$-submodules of $M$ isomorphic to $L(\lambda)$.
Let $M_\lambda$ be the subspace of all highest weight vectors in $M^\lambda$.
We note that $M_\lambda \cong \text{Hom}_\mathfrak{g}(L(\lambda), M)$.
Let $V^{\mathfrak{g}}=\{v \in V \ | \ gv=0 \ \text{for all} \ g \in \mathfrak{g} \}.$
It is easy to show that each $M_\lambda$ is a $V^\mathfrak{g}$-module.


\begin{corollary} \label{Cor6.7}
 We have the following statement.
\begin{enumerate}[{(1)}]

\item For every $\lambda \in P$ with $M_\lambda \neq 0$, $M_\lambda$ is an irreducible $V^\mathfrak{g}$-module;

\item If $M_\lambda$ and $M_\mu$ are nonzero, then $M_\lambda$ and $M_\mu$ are isomorphic $V^\mathfrak{g}$-modules if and only if $\lambda=\mu$.

\end{enumerate}

\end{corollary}

\begin{proof}
Since for any $u \in \mathfrak{g}$, the actions of $u$ on both $V$ and $M$ are all derivation actions,
it follows that $V$ is a $U(\mathfrak{g})$-module vertex algebra, and $M$ is a  Hopf $V$-module.
Now the result follows from Theorem \ref{main6-1}.

\end{proof}

\begin{remark}
If $M=V$, Corollary \ref{Cor6.7}  have been established in \cite{DLM1}.
\end{remark}


\begin{thebibliography}{AAA1}


\bibitem[B]{B}
R. E. Borcherds, Vertex algebras, Kac-Moody algebras, and the Monster,
{\em Proc. Natl. Acad. Sci. USA} {\bf 83} (1986), 3068-3071.



\bibitem[DLM1]{DLM1}
C. Dong, H. Li, and G. Mason, Compact automorphism groups of vertex operator algebras,
{\em IMRN} {18} (1996), 913-921.

\bibitem[DLM2]{DLM2}
C. Dong, H. Li, and G. Mason, Vertex operator algebras and associative algebras,
{\em J. Algebra.} {\bf 206} (1998), 67--96.

\bibitem[DM]{DM}
C. Dong and G. Mason,
On quantum Galois theory,
{\em Duke Math. J.} {\bf 86} (1997), 305--321.

\bibitem[DRY]{DRY}
C. Dong, L. Ren, C. Yang, Orbifold theory for vertex algebras and Galois correspondence, arXiv: 2302. 09474.


\bibitem[DW]{DW}
C. Dong and H. Wang, Hopf actions on vertex operator algebras, {\em J. Algebra} {\bf 514} (2018), 310-329.

\bibitem[DY]{DY} C. Dong and G. Yamskulna,
Vertex operator algebras, generalized doubles and dual pairs,
{\em Math. Z.} {\bf 241} (2002), 397-423.

\bibitem[DYa]{DYa} C. Dong and C. Yang, $G$-twisted associative algebras of vertex operator superalgebras,
{\em J. Algebra} {\bf 606} (2022), 323-340.



\bibitem[EK]{EK} P. Etingof, D. Kazhdan, Quantization of Lie bialgebras, $V$: Quantum vertex operator
algebras, {\em Selecta Math.} {\bf 6} (2000), 105-130.

\bibitem[EW]{EW} P. Etingof, C. Walton, Pointed Hopf actions on fields, I, {\em Transform. Groups} 20, No.4, (2015), 985-1013.

\bibitem[Li1]{Li1} H. Li, Simple vertex operator algebras are nondegenerate, {\em J. Algebra} {\bf 267} (2003), 199-211.

\bibitem[Li2]{Li2} H. Li, Constructing quantum vertex algebras, {\em Internat. J. Mathematics} {\bf 17} (2006), 441-476.

\bibitem[LL]{LL}
J. Lepowsky, H. Li, Introduction to Vertex Operator Algebras and Their Representations,
{\em Progress in Mathematics}, Vol. {\bf 227}, Springer, 2004.


\bibitem[M]{M}
S. Montgomery, Hopf Algebras and Their Actions on Rings, CBMS Regional Conference Series
in Mathematics, vol. {\bf 82}, AMS, 1993.

\bibitem[MT]{MT} 
M. Miyamoto, K. Tanabe, Uniform product of $A_{g,n}(V)$ for an orbifold
model $V$ and $G$-twisted Zhu algebra, {\em J. Algebra}. {\bf 274} (2004), 80-96.

\bibitem[N]{N}
Warren D. Nichols, Quotients of hopf algebras, {\em Communications in Algebra}, {\bf 6} (1978), {1789-1800}.


\bibitem[T]{Ta}
K. Tanabe, A Schur-Weyl type duality for twisted weak modules over a vertex algebra, arXiv: 2303.15692.




\bibitem[YY]{YaY}
C. Yang, M. L, Yang, Associative algebras of $\Z$-graded vertex operator superalgebras and their applications, {\em J. Algebra} {\bf 633} (2023), 20-42.

\bibitem[Z]{Z} Y. Zhu; Modular invariance of characters of
vertex operator algebras, {\em J. Amer. Math. Soc}. {\bf 9} (1996), 237--302.
\end{thebibliography}
\end{document}